\documentclass[12pt,a4paper,french,english]{amsart}

\usepackage{amsfonts,amsmath,amssymb, amscd,fullpage}
\usepackage{stmaryrd}
\usepackage[all]{xy}
\usepackage[T1]{fontenc}
\usepackage{lmodern}
 \usepackage[utf8]{inputenc}
\usepackage{enumitem}
\usepackage[french, english]{babel}

\newtheorem{theorem}{Theorem}[section]
\newtheorem{lemma}[theorem]{Lemma}
\newtheorem{proposition}[theorem]{Proposition}

\theoremstyle{definition}

\theoremstyle{remark}
\newtheorem{remark}[theorem]{Remark}

\newcommand\pf{\begin{proof}}
\newcommand\epf{\end{proof}}

\newcommand\lHlC{{_H^C\mathcal M}}
\newcommand\lA{{_A\mathcal M}}
\newcommand\lH{{_H\mathcal M}}

\numberwithin{equation}{section}

\title{Faithful flatness of Hopf algebras over coideal subalgebras with a bimodule conditional expectation}

\author{Julien Bichon}
\address{ Universit\'e Clermont Auvergne, CNRS, LMBP, F-63000 CLERMONT-FERRAND, FRANCE}
\email{julien.bichon@uca.fr}


\begin{document}

\begin{abstract}
We give a direct and self-contained proof that if  $H$ is a Hopf algebra and  $A\subset H$ is a right coideal subalgebra such $A$ is a direct summand in $H$ as an $A$-bimodule, then $H$ is faithfully flat as a left and right $A$-module. 
\end{abstract}

\maketitle

\section{Introduction}

The aim of this note is to give a direct and self-contained proof of the following result.

\begin{theorem}\label{thm:ffce}
	Let $H$ be a Hopf algebra and let $A\subset H$ be a right coideal subalgebra. If $A$ is a direct summand in $H$ as an $A$-bimodule, then $H$ is faithfully flat as a left and right $A$-module. 
\end{theorem}

The original author's interest for Theorem \ref{thm:ffce} came from A. Chirvasitu's result \cite[Theorem 2.1]{chi14} on the faithful flatness of a cosemisimple Hopf algebras over its Hopf subalgebras.  Indeed, Chirvasitu's proof is divided in two steps: he first proves the crucial fact that a Hopf subalgebra $A\subset H$ of a cosemisimple Hopf algebra $H$ is a direct summand of $H$ as an $A$-bimodule, and then concludes using \cite[Proposition 1.4, Proposition 1.6]{chi14}, results that are obtained by a discussion involving an important number of external references \cite{brwi,mes06,nus97,tak} and various technologies. It is hoped that the present direct and self-contained proof of Theorem \ref{thm:ffce} will provide   an easier access to the proof of Chirvasitu's Theorem.

I wish to thank B. Mesablishvili for useful comments and remarks on previous versions of this note, and S. Skryabin for pointing out his recent paper \cite{Skr}, which shows that the assumption ``$A$ is a direct summand in $H$ as an $A$-bimodule'' in Theorem \ref{thm:ffce}
cannot be weakened to ``$A$ is direct summand as a right (or left) $H$-module''.

\medskip

\noindent
\textbf{Notations and conventions.}
We work over a fixed field $k$, and  assume that the reader is familiar with the theory of Hopf algebras as e.g. in \cite{mon}.
If $H$ is a Hopf algebra, as usual, $\Delta$, $\varepsilon$ and $S$ stand respectively for the comultiplication, counit and antipode of $H$. We use Sweedler's notations in the standard way. 

The category of left $A$-modules over an algebra is denoted $\lA$, the category of left $C$-comodules over a coalgebra is denoted ${^C\mathcal M}$, etc...  

As usual, we say that a right $A$-module $M$ is  flat if the functor $M\otimes_A- : \lA \to {_k\mathcal M}$ is  exact, which amounts to say map $M\otimes_A-$ preserves injective maps (monomorphisms), and that $M$ is faithfully flat if it is flat and $M\otimes_A-$ creates exact sequences as well. Left (faithfully) flat $A$-modules are defined similarly. We also say that an algebra extension $A\subset B$ is right or left (faithfully) flat is $B$ is (faithfully) flat  as a right or left $A$-module.  

If $A\subset B$ is an algebra extension, then $A$ is direct summand in $B$ as a right $A$-module if and only there exists a right $A$-linear map $E : B\to A$ such that $E_{|A}={\rm id}_A$, and we call such a map $E$ is a right conditional expectation for the extension $A\subset B$. The notion of left conditional expectation is defined similarly, and a bimodule conditional expectation is an $A$-bimodule map $E : B \to A$ such that  $E_{|A}={\rm id}_A$.

\section{Proof Of Theorem \ref{thm:ffce}\label{sec:rhff}}

\subsection{Preliminary set-up} We begin by fixing a number of notation and constructions, which will run throughout the section. All this material can be found in \cite{tak}.

Let $H$ be a Hopf algebra and let $A\subset H$ be a right coideal subalgebra, which means that $A$ is subalgebra of $H$ such that $\Delta(A)\subset A\otimes H$. Let $A^+ = {\rm Ker}(\varepsilon) \cap A$ and consider $HA^+$, the left ideal of $H$ generated by $A^+$. It is an immediate verification  that $HA^+$ is a coideal in $H$ ($\Delta(HA^+)\subset HA^+\otimes H + H \otimes HA^+$ and $\varepsilon(HA^+)=0$), so we can form the quotient coalgebra $C=H/HA^+$ together with the canonical surjection : $\pi : H \to C=H/HA^+$. The coalgebra $C$ has as well a left $H$-module structure induced by $\pi$, so that $C$ is a left $H$-module coalgebra. We therefore consider the category of (relative) Hopf modules $\lHlC$, whose objects are the left $H$-modules and left $C$-comodules $X$ such that the coaction $\alpha_X : X \to C \otimes X$ is left $H$-linear, i.e. in Sweedler notation, we have for any $h\in H$ and $x\in X$,
$$(h.x)_{(-1)} \otimes (h.x)_{(0)} = h_{(1)}.x_{(-1)} \otimes h_{(2)}.x_{(0)}$$ 
For a left $A$-module $M$, the induced $H$-module $H\otimes_A M$ has a left $C$-comodule structure given by $(h\otimes_A m)_{(-1)} \otimes (h\otimes_A m)_{(0)}= \pi(h_{(1)})\otimes  h_{(2)}\otimes_A m$ making it into an object of $\lHlC$. This defines the induction functor
\begin{align*}
 L = H\otimes_A- : \lA &\longrightarrow \lHlC \\
M & \longmapsto H\otimes_A M
\end{align*}
For an object $X$ in $\lHlC$, let 
$$^{{\rm co} C}X = \{x \in X \ | \ x_{(-1)}\otimes x_{(0)} = \pi(1) \otimes x\}$$ 
It is immediate to check that $^{{\rm co} C}X\subset X$ is a sub-$A$-module and this defines a functor
\begin{align*}
 R = {^{{\rm co} C}}(-) : \lHlC &\longrightarrow \lA \\
X & \longmapsto {^{{\rm co} C}X}
\end{align*}
which is right adjoint to $L$. We therefore have a pair of adjoint functors
\begin{align} 
(L,R) : \lA \rightleftarrows \lHlC \label{eqn:pair}\end{align} 
whose respective unit and counit are given by
\begin{alignat}{3}
 \eta_M : M &\longrightarrow {^{{\rm co}  C}(H\otimes_A M)} \quad \quad \mu_X : &H\otimes_A {^{{\rm co}C} X} &\longrightarrow X \label{eqn:unicouni} \\
m & \longmapsto 1_H \otimes_A m   &h\otimes x &\longmapsto h.x \nonumber
\end{alignat}

We have now the necessary material to state the following result, which is \cite[Proposition 1.6]{chi14}.

\begin{theorem}\label{thm:equff}
	Let $H$ be a Hopf algebra, let $A\subset H$ be a right coideal subalgebra and let $C=H/HA^+$ be the corresponding quotient coalgebra. The following assertions are equivalent.
	\begin{enumerate}
	\item The induction functor ${_A\mathcal M}\longrightarrow\lHlC$ is an equivalence of categories.
	\item The extension $A\subset H$ is right faithfully flat.
	\item The above unit and counit morphisms (\ref{eqn:unicouni})  are isomorphisms.
\end{enumerate}
	\end{theorem}

It is immediate that $(1)\Rightarrow (2)$ since an equivalence of categories is a faithfully exact functor and the exact sequences in $\lA$ and $\lHlC$ are precisely those that are exact in ${_k\mathcal M}$. It is clear that $(3)\Rightarrow (1)$. The arguments we develop to prove Theorem \ref{thm:ffce} will provide as well a proof of (2)$\Rightarrow$(3).

\subsection{The canonical isomorphisms} We will use some ``canonical'' isomorphisms, that we construct in this subsection. 

For a left $H$-module $X$, endow $C\otimes X$ with the tensor product left $H$-module structure and with the left $C$-comodule structure provided by the comultiplication of $C$. In this way $C\otimes X$ becomes an object of $\lHlC$ (in fact $C\otimes X$ is the image of $X$ by the right adjoint to the forgetful functor $\lHlC \to \lH$).

\begin{proposition}\label{prop:cano}
Let $A\subset H$ be a right coideal subalgebra and let $X$ be  left $H$-module. The canonical map
	\begin{align*}
		\kappa_X : H\otimes_A X &\longrightarrow C\otimes X \\
		h\otimes_A x &\longmapsto \pi(h_{(1)}) \otimes h_{(2)}.x
	\end{align*}
is an	isomorphism in the category $\lHlC$.
\end{proposition}

\begin{proof}
It is a direct verification that $\kappa_X$ is a a morphism in $\lHlC$, and that 
\begin{align*}
 C \otimes X &\longrightarrow H\otimes_A X \\
\pi(h) \otimes x & \longmapsto h_{(1)} \otimes_A S(h_{(2)}).x
\end{align*}
is the inverse isomorphism.
\end{proof}

\subsection{The unit of the adjunction} We first analyse the unit of our adjunction, starting with a general observation.

\begin{proposition}\label{prop:flatish}
 Let $A\subset B$ be an algebra extension, let $M$ be  a left $A$-module $M$, and consider the map
	\begin{align*}
	\iota_M : M & \longrightarrow \biggl\{  \sum_i x_i \otimes_A m_i \in B \otimes_A M \ \big|  \ \sum_ix_i\otimes_A 1_B\otimes_A  m_i = \sum_i 1_B\otimes_A x_i\otimes_A  m_i\biggr\} \\
	m & \longmapsto 1_B\otimes_A m 
	\end{align*}
The $\iota_M$ is an isomorphism if one of the following conditions holds:
\begin{enumerate}
		\item $A\subset B$ is right faithfully flat;
		\item $A$ is a direct summand in $B$ as a right $A$-module.
	\end{enumerate}
\end{proposition}

\begin{proof}
 It is a well-known result  (see e.g. the second theorem of Section 13.1 in \cite{wat}) that $\iota_M$ is an isomorphism if $B$ is faithfully flat as a right $A$-module, that we do not reproduce here.

Let  $E : B \to A$ be a right conditional expectation. The right $A$-linearity of $E$ enables us to define, for any left $A$-module $M$,  the map \begin{align*} E_M : B\otimes_A M &\to M \\ x\otimes_A m &\mapsto E(x). m\end{align*}
 For simplicity denote   $X(M)$ the space on the right. Let us check that $E_{M|{X(M)}} : X(M)\to M$ is an inverse isomorphism to $\iota_M$. In one direction it is clear that $E_M \circ \iota_M={\rm id}_M$. To prove that $\iota_M \circ E_{M|X(M)} ={\rm id}_{X(M)}$, similarly to before, notice that the right $A$-linearity of $E$ enables us to define the map \begin{align*} E_M' : B\otimes_A B\otimes _A M &\to B\otimes_A M \\ x\otimes_A y\otimes_A m &\mapsto E(x)y\otimes_A m\end{align*}
For  $\sum_i x_i \otimes_A m_i \in X(M)$, we have 
$$\sum_i 1_B\otimes_A x_i\otimes_A  m_i=\sum_ix_i\otimes_A 1_B\otimes_A  m_i$$ and applying $E'_M$ yields
$$ \sum_i x_i\otimes_A  m_i = \sum_i E(x_i)\otimes_A  m_i =1_B\otimes_A\left(\sum_i E(x_i). m_i\right) = \iota_M \circ E_M\left(\sum_i x_i\otimes_A  m_i\right)$$ 
which concludes the proof.
\end{proof}

\begin{proposition}\label{prop:unitiso}
Let $A \subset H$ be a right coideal subalgebra.
Assume that $A\subset H$ is right faithfully flat or that 
		$A$ is a direct summand in $H$ as a right $A$-module. Then for any left $A$-module $M$, the unit map
	\begin{align*}
	\eta_M : M \longrightarrow {^{{\rm co}  C}(H\otimes_A M)}
	\end{align*}
	is an isomorphism.
\end{proposition}

\begin{proof}
We consider the canonical isomorphism
\begin{align*}
\kappa_M'=\kappa_{H\otimes_A M} : H\otimes_A (H\otimes_A M)&\longrightarrow C\otimes (H\otimes_A M) \\
h\otimes_A h'\otimes_A m &\longmapsto \pi(h_{(1)}) \otimes h_{(2)}h' \otimes_A m
\end{align*}
	 from Proposition \ref{prop:cano}. For $\sum_i h_i\otimes_A m_i \in  {^{{\rm co}  C}(H\otimes_A M)}$, we have 
	 $$\kappa_M'\left(\sum_i h_i\otimes_A 1_H \otimes_A m_i\right) = \sum_i \pi(1)\otimes h_i\otimes_A m_i=\kappa'_M\left(\sum_i 1_H\otimes_A h_i\otimes_A m_i\right)$$
	 and the  injectivity of $\kappa_M'$ gives
	 $$\sum_i h_i\otimes_A 1_H \otimes_A m_i=\sum_i 1_H\otimes_A h_i\otimes_A m_i$$
	 Our assumption then ensures, by Proposition \ref{prop:flatish}, the existence of a unique $m \in M$ such that $\sum_i h_i\otimes_A m_i = 1_H\otimes_Am$. This therefore defines a map 
	 ${^{{\rm co}  C}(H\otimes_A M)} \to M$, which is clearly an inverse to $\eta_M$.
	\end{proof}

\begin{remark}
If $A$ is a direct summand in $H$ as a right $A$-module, Proposition \ref{prop:unitiso} ensures in particular that $^{{\rm co} C}H =A$. Hence in view of Proposition \ref{prop:cano}, we see that $A\subset H$ is coalgebra Galois extension over $C$. Once this is noticed, the shortest way to obtain a proof of Theorem \ref{thm:ffce} is certainly to invoke \cite[Proposition 4.4]{brz}. I thank B. Mesablishvili for pointing out coalgebra Galois extension in this context. 
\end{remark}

\subsection{The counit of the adjunction} We now analyse the counit of our adjunction $(L,R)$. We begin with a lemma, in which we use the following notation: if $X$ is an object of $\lHlC$,  we denote $i_X : {^{{\rm co} C}X}\to X$ the natural inclusion map.

\begin{lemma}\label{lem:proj}
	Let $A\subset H$ be a right coideal subalgebra and let $X$ be an object of  $\lHlC$. Assume that $A$ is a direct summand in $H$ as a right $A$-module, and let $E : H \to A$ be a right conditional expectation. Then the map $$p_X = (E\otimes_A{\rm id}_X)\circ \kappa_X^{-1}\circ\alpha_X : X \to X$$ is a projection of $X$ onto ${^{{\rm co}C} X}$. If moreover $E$ is an $A$-bimodule conditional expectation, then $p_X : X \to {^{{\rm co}C} X}$ is $A$-linear, and hence the map ${\rm id}_H\otimes_A i_X : H\otimes_A {^{{\rm co}C} X} \to H\otimes_A X$ is injective.
\end{lemma}

\begin{proof}
The above map $p_X$ is well-defined since $E$ is right $A$-linear, is an $A$-bimodule map as soon as $E$ is,  and for $x\in  {^{{\rm co} C}X}$, it is clear that $p_X(x)=x$, since $E(1_H)=1_A$. Thus one just has to check that $p_X$ has values into ${^{{\rm co} C}X}$, which follows from the commutativity of the following diagram:
$$\xymatrix@C=2.2cm@R=0.8cm{ X \ar[r]^{\alpha_X} \ar[d]^{\alpha_X}& C\otimes X \ar[r]^{\kappa_X^{-1}} \ar[d]^{{\rm id}_C\otimes \alpha_X} & H\otimes_AX \ar[r]^{E\otimes_A {\rm id}_X} \ar[d]^{{\rm id}_H\otimes_A\alpha_X}& X  \ar[d]^{\alpha_X}  \\
C\otimes X \ar[r]^{\Delta_C  \otimes {\rm id}_X} \ar[rd]_{\kappa_X^{-1}} & C\otimes (C \otimes X) \ar[r]^{\kappa_{C\otimes X}^{-1}} & H\otimes_A(C\otimes X) \ar[r]^-{E\otimes_A \otimes {\rm id}_C \otimes _X} & C\otimes X \\
 & H\otimes_A X \ar[ru]_{\nu_X}\ & & 
}$$ 
where $\nu_X : H\otimes_A X : H\otimes A (C\otimes X)$ is defined by $\nu_X(h \otimes_A x) = h\otimes_A\pi(1)\otimes x$.

If $E$ is an $A$-bimodule map, then $p_X$ is left $A$-linear and therefore ${\rm id}_H \otimes_A p_X$ provides a retraction to ${\rm id}_H\otimes_A i_X$, which gives the last statement.
\end{proof}

\begin{proposition}\label{prop:mu}
		Let $A\subset H$ be a right coideal subalgebra and let $X$ be an object of  $\lHlC$. Consider the following assertions:
\begin{enumerate}[label=(\alph*)]
 \item  ${\rm id}_H\otimes_A i_X : H\otimes_A {^{{\rm co} C} X} \to  H\otimes_A X$ is injective;
\item $\mu_X :  H\otimes_A {^{{\rm co} C} X} \to X$ is injective;
\item $\mu_X :  H\otimes_A {^{{\rm co} C} X} \to X$ is surjective.
\end{enumerate}
Then we have $(a)\iff (b) \Longrightarrow (c)$. These assertions hold true if one of the following conditions is satisfied:
\begin{enumerate}
 \item $H$ is flat as a right $A$-module;
\item $H$ is a direct summand in $B$ as an $A$-bimodule.
\end{enumerate}		
\end{proposition}

\begin{proof}
Consider the map $\delta : X \to C\otimes X$ defined by $\delta(x) =x_{(-1)}\otimes x_{(0)} -\pi(1)\otimes x$. This map is $A$-linear and the sequence of $A$-modules $0 \to {^{{\rm co}C}X} \overset{i_X} \to X \overset{\delta} \to C\otimes X$ is exact. Applying $H\otimes_A-$ yields the  sequence 
$$ H\otimes_A {^{{\rm co}C}X} \overset{{\rm id}_H\otimes_A i_X} \longrightarrow H\otimes_AX \overset{{\rm id}_H\otimes_A \delta} \longrightarrow H\otimes_A (C\otimes X)$$ 
that fits in the commutative diagram
$$\xymatrix@C=1.5cm@R=1cm{ & H\otimes_A {^{{\rm co}C}X} \ar[r]^{{\rm id}_H\otimes_A i} \ar[d]^{\mu_X} & H\otimes_A  X \ar[r]^{{\rm id}_H\otimes_A \delta} \ar[d]^{\kappa_X} & H\otimes_A (C\otimes X)  \ar[d]^{\kappa_{C\otimes X}}  \\
0  \ar[r] & X \ar[r]^{\alpha_X} & C\otimes X \ar[r]^{\nabla} & C\otimes (C\otimes X)
}$$
where $\nabla :  C\otimes X \to C\otimes  (C \otimes X)$ is defined by
$$\nabla(\pi(h) \otimes x)  = \pi(h)\otimes x_{(-1)}\otimes x_{(0)} - \pi(h_{(1)})\otimes \pi(h_{(2)})\otimes x$$ 
Since $\kappa_X^{-1} \circ \alpha_X$ is injective, we get $(a)\iff (b)$.

Assume that (a) holds. Then the upper sequence in the above diagram is exact (while the lower row is exact by construction), and it is a simple diagram chasing to check that $\mu_X$ is surjective.

If (1) holds, then (a) holds by the definition of flatness, and if (2) holds, Lemma \ref{lem:proj} ensures that (a) holds.	
\end{proof}

The proof  of right faithful flatness in  Theorem \ref{thm:ffce}  is now immediate: if $A$ is a direct summand in $H$ as an $A$-bimodule, Proposition \ref{prop:unitiso} ensures that the unit of the adjunction $(L,R)$ is an isomorphism, and Proposition \ref{prop:mu} ensures that the counit is an isomorphism as well, so $B\subset H$ is right faithfully flat by (3)$\Rightarrow$(2) in Theorem \ref{thm:equff}.

Under the assumption that $A$ is a direct summand in $H$ as an $A$-bimodule,  left faithful flatness is shown similarly by considering the right $H$-module quotient coalgebra $D = H /A^+H$, and the category $\mathcal M_H^D$, we do not write the details.

Notice as well  that Proposition \ref{prop:unitiso}  and Proposition \ref{prop:mu} combined together immediately show (2)$\Rightarrow$(3) in Theorem \ref{thm:equff}.

\section{Concluding remark}

In the recent paper \cite{Skr}, S. Skryabin shows that for any finite-dimensional nonsemisimple Hopf algebra $A$ there exists a Hopf algebra $H$ containing $A$ as a Hopf subalgebra such that $H$ is not flat over A. This shows that the assumption 
``$A$ is a direct summand in $H$ as an $A$-bimodule'' in Theorem \ref{thm:ffce}
cannot be weakened to ``$A$ is direct summand as a right (or left) $H$-module'', because since a finite-dimensional Hopf algebra is self-injective (projective modules are injective),   $A$ is direct summand as a right (or left) $H$-module.

We conclude by summarizing what is known about flatness when $A$ is direct summand as a right $H$-module.

\begin{proposition}
	Let $A\subset H$ be a right coideal subalgebra. Assume that $A$ is a direct summand in $H$ as a right $A$-module. Then the following assertions are equivalent:
	\begin{enumerate}
		\item $H$ is faithfully flat as a right $A$-module;
		\item $H$ is flat as a right $A$-module;
		\item For any object   $X$ of  $\lHlC$, the map  ${\rm id}_H\otimes_A i_X : H\otimes_A {^{{\rm co} C} X} \to  H\otimes_A X$ is injective;
		\item  For any object   $X$ of  $\lHlC$, the map $\mu_X :  H\otimes_H {^{{\rm co} C} X} \to X$ is surjective.
	\end{enumerate}
\end{proposition}

\begin{proof}
	By definition (1)$\Rightarrow$(2) and (2)$\Rightarrow$(3),  while (3)$\Rightarrow(4)$ follows from Proposition \ref{prop:mu}. Assume that (4) holds. Since $A$ is a direct summand in $B$ as a right $A$-module, by Proposition \ref{prop:unitiso} we are in the situation of a pair of adjoint functors $(L,R)$ whose unit is an isomorphism and counit is an epimorphism: it is then easy to check that $R$ faithful, and that the counit is a monomorphism as well, so that $L$ and $R$ are inverse equivalences (since we are dealing with categories in which morphisms that are both monomorphisms and epimorphisms are isomorphisms). Hence $H$ is faithfully flat as a right $A$-module.
\end{proof}

\end{document}